\documentclass[11pt]{article}

\RequirePackage{amsmath,amssymb}
\RequirePackage{amsthm}
\RequirePackage{natbib}
\RequirePackage[colorlinks,allcolors=blue]{hyperref}

\usepackage{fullpage}
\usepackage{setspace}
\usepackage{bbm,graphicx,bm} %
\usepackage{enumitem}
\usepackage{multirow}
\usepackage[all,import]{xy}
\usepackage{appendix}
\usepackage{comment}
\usepackage{color}
\usepackage{soul}
\usepackage{todonotes}
\usepackage{authblk}

\theoremstyle{definition}
\newtheorem{assumption}{Assumption}

\newtheorem{remark}{Remark}
\newtheorem{theorem}{Theorem}

\newtheorem{example}{Example}
\newtheorem{definition}{Definition}
\newtheorem{corollary}{Corollary}
\newtheorem{proposition}{Proposition}

\newcommand\indep{\protect\mathpalette{\protect\independenT}{\perp}}
\def\independenT#1#2{\mathrel{\rlap{$#1#2$}\mkern2mu{#1#2}}}
\newcommand{\ATE}{\tau^\textup{ATE}}

\newcommand{\var}{\mathrm{var}}
\newcommand{\Var}{\mathrm{var}}

\def\b1{\boldsymbol{1}}

\def\d{\text{d}}

\newcommand{\mc}{\mathcal}
\newcommand{\E}{\mathrm{E}}

\definecolor{RED}{RGB}{255,0,0}

\usepackage{soul}

\begin{document}

\title{{\huge Overlap in observational studies\\with high-dimensional covariates}\thanks{We thank Xiaohong Chen, Skip Hirshberg, Elie Tamer, participants at the Atlantic Causal Inference Conference, Stanford University, and Yale University for helpful comments and discussions. PD thanks the National Science Foundation, Grant DMS 1713152. JS thanks Office of Naval Research (ONR) Grants N00014-17-1-2176 and N00014-15-1-2367.}\\[1em] }

\author[a]{Alexander D'Amour\footnote{Corresponding author. Alexander D'Amour is now a Research Scientist at Google Research, Cambridge, MA USA; this work was completed while he was a Neyman Visiting Assistant Professor at UC Berkeley.}}
\author[a]{Peng Ding}
\author[a]{Avi Feller}
\author[a]{Lihua Lei\footnote{Lihua Lei is now a postdoctoral fellow in the Department of Statistics at Stanford University, Stanford, CA USA. This work was completed while he was a graduate student in the Department of Statistics at UC Berkeley.}}
\author[a]{Jasjeet Sekhon}

\affil[a]{UC Berkeley Department of Statistics, Evans Hall, Berkeley, CA USA}
\affil[ ]{\tt alexdamour@google.com, pengdingpku@berkeley.edu, afeller@berkeley.edu, lihualei@stanford.edu, sekhon@berkeley.edu}

\date{}

\maketitle

\pagenumbering{gobble}

\begin{abstract}
\singlespacing
\noindent Estimating causal effects under exogeneity hinges on two key assumptions: unconfoundedness and overlap. Researchers often argue that unconfoundedness is more plausible when more covariates are included in the analysis. Less discussed is the fact that covariate overlap is more difficult to satisfy in this setting. In this paper, we explore the implications of overlap in observational studies with high-dimensional covariates and formalize curse-of-dimensionality argument, suggesting that these assumptions are stronger than investigators likely realize. Our key innovation is to explore how strict overlap restricts global discrepancies between the covariate distributions in the treated and control populations. Exploiting results from information theory,  we derive explicit bounds on the average imbalance in covariate means under strict overlap and show that these bounds become more restrictive as the dimension grows large. We discuss how these implications interact with assumptions and procedures commonly deployed in observational causal inference, including sparsity and trimming.
\end{abstract}

\vspace{2em}
\begin{small}
\end{small}

\clearpage
\pagenumbering{arabic}
\onehalfspacing

\section{Introduction}

Accompanying the rapid growth in administrative databases and online platforms, there has been a push to extend methods for estimating causal effects under exogeneity to settings with high-dimensional covariates \citep{Belloni2014, Farrell2015, Athey2016}.
These studies typically require a pair of identifying assumptions \citep{Rosenbaum1983, imbens2004nonparametric}: \emph{unconfoundedness}, also known as selection on observables, in which the treatment assignment mechanism depends only on observed covariates; and \emph{overlap}, also known as positivity or common support, in which all units have a non-zero probability of assignment to each treatment condition.

A key argument for high-dimensional observational studies is that unconfoundedness is more plausible when the analyst adjusts for more covariates \citep{Rosenbaum2002,Rubin2009}.
Setting aside notable counter-examples to this argument~\citep{Pearl2011, wooldridge2016zbias}, %
the intuition is straightforward to state: the richer the set of covariates, the more likely that unmeasured confounding variables become measured confounding variables. 
This intuition, however, has the opposite implications for overlap: the
richer the set of covariates, the closer these covariates come to perfectly predicting treatment assignment for at least some subgroups.

We formalize this curse of dimensionality argument and demonstrate that there are strong implications of overlap when there are many covariates.
In particular, we focus on the strict overlap assumption,
which asserts that the propensity score is bounded away from zero and one with probability one.
While this appears to be a local constraint,
we show that strict overlap implies global restrictions on the discrepancy between the covariate distributions in the treated and control populations.
To do so, we re-frame strict overlap as bounding a likelihood ratio, which is a well-studied problem in information theory \citep{Hellman1970}.
Adapting results from \citet{Rukhin1997}, we derive explicit bounds on various types of covariate imbalance, and show that these bounds become more restrictive as the dimension of the covariates grows.
For example, we show that as the dimension of the covariates grows, strict overlap implies that the covariates must either be highly correlated, or that their means must become arbitrarily close to balance on average.
To put these results into context, we discuss how the implications of strict overlap intersect with common modeling assumptions, and how our results inform the common practice of trimming in high-dimensional contexts.

We contribute to a growing literature on the critical role of overlap in observational settings.
In the context of semiparametric estimators, several papers show that the convergence rate critically depends on the level of overlap \citep{Khan2010, hong2018inference, ma2018robust}; see \citet{busso2014new} for relevant simulation evidence.
Recognizing this, one common approach is to trim units that have extreme values of the propensity score \citep{dehejia1999causal, Crump2009, Petersen2012, yang2018asymptotic}.
An alternative is to instead propose estimators and inference methods that have additional robustness to overlap violations \citep{chen2008semiparametric, VanderLaan2011, chaudhuri2014heavy, rothe2017robust, armstrong2017finite, sasaki2017inference}.
Finally, our results are especially relevant for recent efforts to incorporate machine learning into estimating causal effects, partly to exploit rich covariates \citep[see][for recent reviews]{Chernozhukov2016, athey2019mlreview}.
On the one hand, by using machine learning to perform covariate adjustment, these methods can achieve parametric convergence rates under extremely weak nonparametric modeling assumptions.
On the other hand, the cost of this nonparametric flexibility is that these methods are highly sensitive to poor overlap. 
Thus, understanding the implications of overlap with high-dimensional covariates is therefore critical across many open research areas.

The paper proceeds as follows. 
Section \ref{sec:framework} sets up the problem and defines key notation.
Section \ref{sec:implications_main} gives the main results on implications of strict overlap.
Section \ref{sec:outcome_models_main} discusses the role of assumptions on the outcome model, such as sparsity, as well as trimming.
Section \ref{sec:future work} offers some discussion.
In separate work, we address possible remedies and methodologies for assessing overlap in this setting, but believe that characterizing the implications of overlap remains of independent interest.

\section{Preliminaries}
\label{sec:framework}

We focus on an observational study with a binary treatment.
For each sampled unit $i$, $(Y_i(0), Y_i(1))$ are potential outcomes, $T_i$ is the treatment indicator, and $X_i$ is the set of covariates.
Let $\{(Y_i(0), Y_i(1)), T_i, X_i\}_{i=1}^n$ be independently and identically distributed according to a superpopulation probability measure $P$. We drop the $i$ subscript when discussing population stochastic properties of these quantities.
We observe triples $(Y^\textup{obs}, T, X)$ where $Y^\textup{obs} = (1-T)Y(0) + TY(1)$.
We would like to estimate the average treatment effect
$$
\tau^\textup{ATE} = \E\{ Y(1) - Y(0) \},
$$
\noindent though our results immediately extend to other estimands like the Average Treatment Effect on the Treated.

The standard approach in observational studies is to argue that identification is plausible conditional on a possibly large set of covariates \citep{Rosenbaum1983, imbens2004nonparametric}.
Specifically, the investigator chooses a set of $p$ covariates $X_{1:p} \subset X$, and assumes unconfoundedness.
\begin{assumption}[Unconfoundedness]
\label{assn:unconfounded}
$
(Y(0), Y(1)) \indep T \mid X_{1:p}.
$
\end{assumption}

Assumption \ref{assn:unconfounded} ensures
\begin{align}
\tau^{\textup{ATE}} &= \E\big[\E\{ Y(1) \mid X_{1:p}\} - \E\{Y(0) \mid X_{1:p}\} \big]\nonumber\\
&= \E\big[\E\{ Y^\textup{obs} \mid T=1, X_{1:p}\} - \E\{Y^\textup{obs} \mid T=0, X_{1:p}\}\big].\label{eq:ident}
\end{align}

Importantly, the conditional expectations in \eqref{eq:ident} are non-parametrically identifiable only if the following population overlap assumption is satisfied. Let $e(X_{1:p}) = P(T = 1 \mid X_{1:p})$ be the propensity score. 
\begin{assumption}[Population overlap]
\label{assn:overlap}
$
0 < e(X_{1:p}) < 1 
$
with probability 1.
\end{assumption}

Assumption~\ref{assn:overlap} is sufficient for non-parametric identification of $\ATE$, but is not sufficient for efficient semiparametric estimation of $\ATE$, a fact we discuss in further detail in the next section.
For this reason, investigators typically invoke a stronger variant of Assumption \ref{assn:overlap} \citep[e.g.,][]{hirano2003efficient, Khan2010}, which we call the strict overlap assumption with bound $\eta$.

\begin{assumption}[Strict overlap]
\label{assn:strict overlap}
For some constant $\eta \in (0, 0.5)$,
$
\eta \leq e(X_{1:p}) \leq 1-\eta  
$
with probability 1.
\end{assumption}

Strict overlap is integral across a range of settings. 
Without any restrictions on the outcome distribution, strict overlap is a necessary condition for the existence of regular semiparametric estimators of $\ATE$ that are uniformly $n^{1/2}$-consistent over a nonparametric model family \citep{Khan2010}. This necessity may not hold if other conditions, e.g., conditional moment conditions and smoothness conditions, are imposed on the potential outcomes \citep[e.g.,][]{chen2008semiparametric, hirshberg2017augmented, ma2018robust}. Technically, we can relax Assumption \ref{assn:strict overlap}, but this will involve non-standard asymptotic analyses \citep[e.g.,][]{hong2018inference, ma2018robust} and it is difficult, if not impossible, to conduct uniform inference on $\ATE$ \citep[e.g.][]{khan2013uniform}.
Nevertheless, a large body of literature assumes strict overlap, even in the presence of outcome restrictions, as it facilitates theoretical analysis; see, for example, \citet{VanderLaan2011} and \citet{Chernozhukov2016}. Moreover, as \citet[][Section 4.1]{khan2013uniform} observe, it is not clear how to conduct uniform inference without strict overlap conditions, except in corner cases.
Indeed, \citet{khan2013uniform} prove that neither bootstrap inference nor pivotal inference is asymptotically valid without this assumption. \citet{ma2018robust} shed some light on the possibility of uniform inference under assumptions on tail behaviors of inverse propensity scores though they do not provide a complete recipe. In general, a lack of uniform inference is problematic in practice, even if we can characterize the limiting behavior for every data generating distribution in a model, because the correct choice of inferential procedure will depend on the unknown truth. See, e.g., \citet{romano1999subsampling}, \citet{andrews2012estimation, andrews2013maximum}, and \citet{chenlikelihood2011} for discussion in other contexts.

\section{Implications of Strict Overlap}
\label{sec:implications_main}

\subsection{Framework}
In this section, we show that strict overlap restricts the overall discrepancy between the treated and control covariate measures, and that this restriction becomes more binding as the dimension $p$ increases.
Formally, we write the control and treatment measures for covariates, for all $p$, as:
\begin{align*}
P_0(X_{1:p} \in A) &:= P(X_{1:p} \in A \mid T=0),\\
P_1(X_{1:p} \in A) &:= P(X_{1:p} \in A \mid T=1).
\end{align*}
For the remainder of the paper, we will assume that the marginal probability that any unit is assigned to treatment, $\pi := P(T = 1)$, is bounded by $\eta \leq \pi \leq 1-\eta$. 
With a slight abuse of notation, we define the marginal probability measure on covariates, implied by the superpopulation distribution, as $
P = \pi P_1 + (1-\pi) P_0,
$
a mixture of the condition-specific probability measures $P_0$ and $P_1$.

We write the densities of $P_1$ and $P_0$ with respect to the dominating measure $P$ as $\d P _1/\d P $ and $\d P _0/\d P $.
We write the marginal probability measures of finite-dimensional covariate sets $X_{1:p}$ as $P_0(X_{1:p})$ and $P_1(X_{1:p})$, and the marginal densities as $\d P _1/\d P (X_{1:p})$ and $\d P _0/\d P (X_{1:p})$.
When discussing density ratios, we will omit the dominating measure $\d P $.

By Bayes' Theorem, Assumption~\ref{assn:strict overlap} is equivalent to the following bound on the density ratio between $P_1$ and $P_0$, which we will refer to as a likelihood ratio:
\begin{align}
b_{\min}  \leq \frac{\d P _1(X_{1:p})}{\d P _0(X_{1:p})} \leq  
 b_{\max} ,\label{eq:bmin bmax}
\end{align}
where 
\begin{align}\label{eq:bmin-bmax-def}
b_{\min} := \frac{1-\pi}{\pi} \frac{\eta}{1-\eta},\qquad
b_{\max} := \frac{1-\pi}{\pi} \frac{1-\eta}{\eta} . 
\end{align}

Implications of bounded likelihood ratios are well-studied in information theory \citep{Hellman1970,Rukhin1993,Rukhin1997}.
Each of the results that follow are applications of a theorem due to \citet{Rukhin1997}, which relates likelihood ratio bounds of the form \eqref{eq:bmin bmax}
to upper bounds on certain divergences measuring the discrepancy between the distributions $P_0(X_{1:p})$ and $P_1(X_{1:p})$.
We include an adaptation of Rukhin's theorem in the appendix, as Theorem~\ref{thm:f bound}. 
We also derive additional implications of this result in the appendix.

In the subsequent, we explore the implications of Assumption~\ref{assn:strict overlap} when there are many covariates.
To do so, we set up an analytical framework in which the covariate sequence $X$ is a stochastic process $(X_{(k)})_{k > 0}$.
For any single problem, the investigator selects a finite set of covariates $X_{1:p}$ from the infinite pool of covariates $(X_{(k)})_{k > 0}$.
Importantly, this framework includes no notion of sample size because we are examining the population-level implications of an assumption about the population measure $P$.
Our results are independent of the number of samples that an investigator might draw from this population.

\begin{remark}[Strict Overlap and Gaussian Covariates]
    \label{rem:Gaussian}
  While we focus on the implications of strict overlap in high dimensions, this assumption also has surprising implications in low dimensions.
  For example, if $X$ is one-dimensional and follows a Gaussian distribution under both $P_0$ and $P_1$, strict overlap implies that $P_0 = P_1$, or that the covariate is perfectly balanced.
    This is because if $P_0 \neq P_1$, the log-density ratio $\log \d P _0/\d P _1(X)$ diverges for values of $X$ with large magnitude, implying that $e(X)$ can be arbitrarily close to 0 or 1 with positive probability.
    Similar results can be derived when $X_{1:p}$ is multi-dimensional Gaussian.
    Thus, for Gaussianly distributed covariates, the implications of strict overlap are so strong that they are uninteresting.
    For this reason, we do not give any examples of the implications of the strict overlap assumption when the covariates are Gaussian. 
\end{remark}

\subsection{Strict Overlap Implies Bounded Mean Discrepancy}

\label{sec:general mean discrepancy}
We now use these bounds to derive concrete implications of strict overlap.
Here, we show that strict overlap implies a strong restriction on the discrepancy between the means of $P_0(X_{1:p})$ and $P_1(X_{1:p})$. 
In particular, when $p$ is large, strict overlap implies that either the covariates are highly correlated under both $P_0$ and $P_1$, or the average discrepancy in means across covariates is small.

We represent the expectations and covariance matrices of $X_{1:p}$ under $P_0$ and $P_1$ as follows:
\begin{align*}
\mu_{0,1:p} := (\mu_{0,(1)}, \ldots, \mu_{0,(p)}) &:= \E_{P_0}(X_{1:p}) ,
&\Sigma_{0,1:p} &:= \var_{P_0}(X_{1:p}) , \\ 
\mu_{1,1:p} := (\mu_{1,(1)}, \ldots, \mu_{1,(p)}) &:= \E_{P_1}(X_{1:p}),
&\Sigma_{1,1:p} &:= \var_{P_1}(X_{1:p}).
\end{align*}
We use $\|\cdot\|$ to denote the Euclidean norm of a vector, and $\|\cdot\|_{\textup{op}}$ to denote the operator norm of a matrix.

\begin{theorem}
\label{thm:general mean discrepancy}
Assumption \ref{assn:strict overlap} implies
\begin{align}
\|\mu_{0,1:p}-\mu_{1,1:p}\| \leq
\min\Big\{&\|\Sigma_{0,1:p}\|^{1/2}_{\textup{op}} \cdot  B_{\chi^2(1\|0)}^{1/2},
\hspace{2mm}\|\Sigma_{1,1:p}\|^{1/2}_{\textup{op}} \cdot B_{\chi^2(0\|1)}^{1/2}\Big\},
\label{eq:general mean discrepancy}
\end{align}
where $b_{\min}$ and $b_{\max}$ are defined in \eqref{eq:bmin-bmax-def}, and 
$$
B_{\chi^2(1\|0)} = (1-b_{\min})(b_{\max}-1),\qquad 
B_{\chi^2(0\|1)} = (1-b_{\max}^{-1})(b_{\min}^{-1} -1)
$$ 
are free of $p$.
\end{theorem}
The proof is included in the Appendix.
Theorem~\ref{thm:general mean discrepancy} has strong implications when $p$ is large.
These implications become apparent when we examine how much each covariate mean can differ, on average, under \eqref{eq:general mean discrepancy}.

\begin{corollary}
\label{cor:general mean MAD}
Assumption \ref{assn:strict overlap} implies
\begin{align}
 p^{-1} \sum_{i=1}^p \left|\mu_{0,(k)} - \mu_{1,(k)}\right|  \leq  p^{-1/2} \min\Big\{&\|\Sigma_{0,1:p}\|^{1/2}_{\textup{op}} \cdot  B_{\chi^2(1\|0)}^{1/2} ,\hspace{2mm}
\|\Sigma_{1,1:p}\|^{1/2}_{\textup{op}} \cdot  B_{\chi^2(0\|1)}^{1/2}\Big\}.
\label{eq:general mean MAD}
\end{align}
\end{corollary}

The mean discrepancy bounds in Theorem~\ref{thm:general mean discrepancy} and Corollary~\ref{cor:general mean MAD} depend on the operator norms of the covariance matrices $\Sigma_{0,1:p}$ and $\Sigma_{1,1:p}$.
The operator norm is equal to the largest eigenvalue of the covariance matrix and is a proxy for the degree to which the covariates $X_{1:p}$ are correlated.
In particular, the operator norm is large relative to the dimension $p$ if and only if a large proportion of the variance in $X_{1:p}$ is contained in a low-dimensional projection of $X_{1:p}$.
For example, in the cases where the components of $X_{1:p}$ are independent, or where $X_{1:p}$ are samples from a stationary ergodic process, the operator norm scales like a constant in $p$.
On the other hand, in the case where the variance in $X_{1:p}$ is dominated by a low-dimensional latent factor model, the operator norm scales linearly in $p$.
We treat these examples precisely in the appendix.

Corollary~\ref{cor:general mean MAD} establishes that strict overlap implies that the average mean discrepancy across covariates is not too large relative to the operator norms of the covariance matrices $\Sigma_{0,1:p}$, and $\Sigma_{1,1:p}$.
When $p$ is large, these implications are strong.
To explore this, let $(X_{(k)})_{k > 0}$ be a sequence of covariates such that for each $p$, $X_{1:p} \subset (X_{(k)})_{k > 0}$. 
When the smaller operator norm $\min ( \|\Sigma_{0,1:p}\|_{\textup{op}},
\|\Sigma_{1,1:p}\|_{\textup{op}} )$ grows more slowly than $p$,
the bound in \eqref{eq:general mean MAD} 
converges to zero, implying that the covariate means are, on average, arbitrarily close to balance. 
On the other hand, for the bound to remain non-zero as $p$ grows large, both operator norms must grow at the same rate as $p$.
This is a strong restriction on the covariance structure; it implies that all but a vanishing proportion of the variance in $X_{1:p}$ concentrates in a finite-dimensional subspace under both $P_0$ and $P_1$. 

\begin{remark}
Theorem~\ref{thm:general mean discrepancy} bounds the mean discrepancy of $X_{1:p}$, which is a special case of a bound on any functional discrepancy of the form $\left|\E_{P_0}\{ g(X_{1:p})\}  - \E_{P_1}\{ g(X_{1:p}) \} \right|$ for any function $g : \mathbb R^p \mapsto \mathbb R$
that is measurable and square-integrable under $P_0$ or $P_1$.
This result is of independent interest, and is included in the appendix. 
\end{remark}

\subsection{Strict Overlap Restricts General Distinguishability}
\label{sec:no good classifier}

In addition to bounds on mean discrepancies, strict overlap also implies restrictions on more general discrepancies between $P_0(X_{1:p})$ and $P_1(X_{1:p})$.
In this section, we present two additional results showing that strict overlap restricts how well the covariate distributions can be distinguished from each other.

First, we show that Assumption~\ref{assn:strict overlap} restricts the extent to which $P_0(X_{1:p})$ can be distinguished from $P_1(X_{1:p})$ by any classifier or statistical test.
Let $\phi(X_{1:p})$ be a classifier that maps from the covariate support $\mc X_{1:p}$ to $\{0,1\}$. We have the following upper bound on the accuracy of any classifier $\phi(X_{1:p})$ when Assumption~\ref{assn:strict overlap} holds.
\begin{proposition}
\label{prop:no good classifier}
Let $\phi(X_{1:p})$ be an arbitrary classifier of $P_0(X_{1:p})$ against $P_1(X_{1:p})$. Assumption \ref{assn:strict overlap} implies the following upper bound on the accuracy of $\phi(X_{1:p})$:
$$
P(\phi(X_{1:p}) = T) \leq 1-\eta.
$$
\begin{proof}
Let 
\begin{equation}\label{eq::bayes-classifier}
\tilde \phi(X_{1:p}) = I\{  e(X_{1:p}) \geq 0.5 \}
\end{equation} 
be the Bayes optimal classifier. The probability of a correct decision from the Bayes optimal classifier is 
\begin{align*}
P(\tilde \phi(X_{1:p}) = T) &= \E\left[  P\{ \tilde \phi(X_{1:p}) = T\mid e(X_{1:p}) \}   \right] \\
&= \E\left[   I\{  e(X_{1:p}) \geq 0.5 \} e(X_{1:p}) + I\{  e(X_{1:p}) < 0.5 \} \{1-e(X_{1:p}) \}  \right] \\
&=\E\left[  \max\left\{e(X_{1:p}), 1-e(X_{1:p})\right\} \right].
\end{align*}
Assumption~\ref{assn:strict overlap} immediately implies $P(\tilde \phi(X_{1:p}) = T) \leq	 1-\eta$. 
The conclusion follows because the Bayes optimal classifier $\tilde \phi(X_{1:p})$ has the highest accuracy among all classifiers based on the covariate set $X_{1:p}$ \citep[Theorem 2.1]{Devroye1996}.
\end{proof}
\end{proposition}

Asymptotically, by Proposition~\ref{prop:no good classifier}, strict overlap implies that there exists no consistent classifier of $P_0$ against $P_1$ in the large-$p$ limit.

\begin{definition}
\label{def:p-consistent}
A classifier $\phi(X_{1:p})$ is \emph{$p$-consistent} if and only if $P(\phi(X_{1:p}) = T) \rightarrow  1$ as $p$ grows to infinity.
\end{definition}

\begin{corollary}[No Consistent Classifier]
\label{cor:no consistent}
Let $(X_{(k)})_{k > 0}$ be a sequence of covariates, and for each $p$, let $X_{1:p}$ be a finite subset.
If Assumption~\ref{assn:strict overlap} holds as $p$ grows large, there exists no $p$-consistent test of $P_0$ against $P_1$.
\end{corollary}

We can characterize the relationship between the dimension $p$ and the distinguishability of $P_0(X_{1:p})$ from $P_1(X_{1:p})$ non-asymptotically by examining the Kullback--Leibler divergence.
The following result is a special case of Theorem~\ref{thm:f bound}, included in the appendix.

\begin{proposition}[KL Divergence Bound]
\label{prop:KL bound}
Assumption \ref{assn:strict overlap} implies 
$$
\textup{KL}(P_1(X_{1:p}) \| P_0(X_{1:p}))
\leq  B_{\textup{KL}(1\|0)},\qquad
\textup{KL}(P_0(X_{1:p}) \| P_1(X_{1:p}))
\leq  B_{\textup{KL}(0\|1)}, 
$$
where
\begin{align*}
B_{\textup{KL}(1\|0)} &:= \frac{(1-b_{\min})b_{\max}\log b_{\max} + (b_{\max}-1)b_{\min}\log b_{\min}}
{b_{\max} - b_{\min}} , \\
B_{\textup{KL}(0\|1)} &:= -\frac{(1-b_{\min})\log b_{\max} + (b_{\max}-1)\log b_{\min}}
{b_{\max} - b_{\min}} 
\end{align*}
are free of $p$, with $b_{\min}$ and $b_{\max}$ defined in \eqref{eq:bmin-bmax-def}. 
\end{proposition}

In the case of balanced treatment assignment with $\pi = 0.5$,  
$B_{\textup{KL}(1\|0)}$ and $B_{\textup{KL}(0\|1)}$ have a simple form:
$$
B_{\textup{KL}(1\|0)} = B_{\textup{KL}(0\|1)} = (1-2\eta) \left| \log \frac{\eta}{1-\eta} \right|.
$$

Proposition~\ref{prop:KL bound} becomes more restrictive for larger values of $p$.
This follows because neither bound in Proposition~\ref{prop:KL bound} depends on $p$, while the KL divergence is free to grow in $p$. 
In particular, by the so-called chain rule, the KL divergence can be expanded into a summation of $p$ non-negative terms \citep[Theorem 2.5.3]{Cover2005}:
\begin{align}
\label{eq:KL accumulate}
\textup{KL}(P_1(X_{1:p}) \| P_0(X_{1:p})) =
\sum_{k=1}^p \E_{P_1} \left\{\textup{KL}(P_1(X_{(k)} \mid X_{1:k-1}) \| P_0(X_{(k)} \mid X_{1:k-1}))\right\}.
\end{align}
Each term in \eqref{eq:KL accumulate} is the expected KL divergence between the conditional distributions of the $k$th covariate $X_{(k)}$ under $P_0$ and $P_1$, after conditioning on all previous covariates $X_{1:k-1}$.
Thus, each term corresponds to the discriminating information added by $X_{(k)}$, beyond the information contained in $X_{1:k-1}$.
In the large-$p$ limit, strict overlap implies that the average unique discriminating information contained in each covariate $X_{(k)}$ converges to zero.

\begin{corollary}
\label{cor:vanishing KL}
Let $(X_{(k)})_{k > 0}$ be a sequence of covariates, and for each $p$, let $X_{1:p}$ be a finite subset of $(X_{(k)})_{k > 0}$.
As $p$ grows large, Assumption \ref{assn:strict overlap} implies
\begin{align}
 p^{-1}\sum_{k=1}^p \E_{P_1} \left\{\textup{KL}(P_1(X_{(k)} \mid X_{1:k-1}) \| P_0(X_{(k)} \mid X_{1:k-1}))\right\} = O(p^{-1}),
\end{align}
and likewise for the KL divergence evaluated in the opposite direction.
\end{corollary}

By Corollary~\ref{cor:vanishing KL}, strict overlap implies that, on average, the conditional distributions of each covariate $X_{(k)}$, given all previous covariates $X_{1:k-1}$, are arbitrarily close to balance. 
In the special case where the covariates $X_{(k)}$ are mutually independent under both $P_0$ and $P_1$, Corollary~\ref{cor:vanishing KL} implies that, on average, the marginal treated and control distributions for each covariate $X_{(k)}$ are arbitrarily close to balance.

\section{Strict Overlap and Modeling Assumptions}
\label{sec:outcome_models_main}

\subsection{Treatment Models: Strict Overlap with Fewer Implications}
In this section, we discuss how the implications of strict overlap align with common modeling assumptions about the assignment mechanism.
We show that certain modeling assumptions already impose many of the constraints that strict overlap implies.
Thus, if one is willing to accept these modeling assumptions, strict overlap has fewer unique implications.

We will focus specifically on the class of modeling assumptions that assert that the propensity score $e(X_{1:p})$ is only a function of a sufficient summary of the covariates $b(X_{1:p})$.
In this case, overlap in the summary $b(X_{1:p})$ implies overlap in the full set of covariates $X_{1:p}$.
Models in this class include sparse models and latent variable models.

\begin{assumption}[Sufficient Condition for Strict Overlap]
\label{assn:sufficient}
There exists some function of the covariates $b(X_{1:p})$ satisfying the following two conditions:
\begin{gather*}
X_{1:p} \indep T \mid b(X_{1:p}),\\ 
\eta \leq e_b(X_{1:p})   \leq 1-\eta,
\end{gather*}
where $e_b(X_{1:p}) := P(T = 1 \mid b(X_{1:p}))$. 
\end{assumption}

Here, the variable $b(X_{1:p})$ is a \emph{balancing score} as in \citet{Rosenbaum1983}.
The propensity score is the coarsest balancing score in the sense that there exists some $h(\cdot)$ such that
$
e(X_{1:p}) = h(b(X_{1:p})).
$
Thus, $b(X_{1:p})$ is a sufficient summary of the covariates $X_{1:p}$ for the treatment assignment $T$,
and
overlap in $b(X_{1:p})$ is a sufficient condition for overlap in the entire covariate set $X_{1:p}$.

\begin{proposition}[Sufficient Condition Statement]
    \label{prop:sufficient}
    Assumption~\ref{assn:sufficient} implies Assumption \ref{assn:strict overlap}. 
    \begin{proof}
The conclusion follows from 
       $e(X_{1:p}) = P(T = 1 \mid X_{1:p}) = \E\{  P(T = 1 \mid X_{1:p}, b(X_{1:p})) \mid X_{1:p} \} = \E\{  P(T = 1 \mid b(X_{1:p})) \mid X_{1:p} \} 
       = \E\{ e_b(X_{1:p}) \mid X_{1:p} \} = e_b(X_{1:p})$.
    \end{proof}
\end{proposition}

Assumption~\ref{assn:sufficient} has some trivial specifications, which are useful examples.
At one extreme, we may specify that
$b(X_{1:p}) = e(X_{1:p})$.
In this case, Assumption~\ref{assn:sufficient} is vacuous: there are no restrictions on the form of the
propensity score; and strict overlap overall is equivalent to strict overlap with respect to $b(X_{1:p})$.
At the other extreme, we may specify $b(X_{1:p})$ to be a constant, i.e., we assume that the data were generated from a randomized trial.
In this case, the overlap condition in Assumption~\ref{assn:sufficient} holds automatically.

Of particular interest are restrictions on $b(X_{1:p})$ between these two extremes, such as the sparse propensity score model in Example~\ref{ex:sparse pscore} below.
Such restrictions trade off stronger modeling assumptions on the
propensity score
$e(X_{1:p})$ with weaker implications of strict overlap.\footnote{These specifications exclude cases such as deterministic treatment rules or treatment assignment in a Regression Discontinuity Design:
even when the covariates are high-dimensional, the information they contain about the treatment assignment is upper bounded by the information contained in $b(X_{1:p})$.}

\begin{example}[Sparse Propensity Score]
\label{ex:sparse pscore}
Consider a study where the propensity score is sparse in the covariate set $X_{1:p}$, so that for some subset of covariates $X_{1:s} \subset X_{1:p}$ with $s < p$,
$$
e(X_{1:p}) = e(X_{1:s}).
$$
This implies
$$
X_{1:p} \indep T \mid X_{1:s}, 
$$
and $e(X_{1:s})$ is a balancing score. 
In this case, strict overlap in the lower-dimensional $X_{1:s}$ implies strict overlap for $X_{1:p}$.
\citet{Belloni2013} and \citet{Farrell2015} propose a specification similar to this, with an ``approximately sparse'' specification for the propensity score.
The approximately sparse specification in these papers is broader than the model defined here, but has similar implications for overlap.
\end{example}

\begin{example}[Latent Variable Model for Propensity Score]
\label{ex:latent variable}
Consider a study where the
treatment assignment mechanism is only a function of some possibly multivariate latent variable $U$, such that
$$
X_{1:p} \indep T \mid U.
$$
For example, such a structure exists when treatment is assigned only as a function of a latent class or latent factor.
In that case, the projection of $e(U) := P(T = 1 \mid U)$ onto $X_{1:p}$ is a balancing score:\footnote{The scalar $b_U(X_{1:p}) := \E\{ e(U) \mid X_{1:p} \}$ is a balancing score because it is equal to the propensity score  $e(X_{1:p}) := P(T = 1 \mid X_{1:p})$. 
Specifically, $e(X_{1:p}) = P(T = 1 \mid X_{1:p}) = \E\{  P(T = 1 \mid X_{1:p}, U) \mid X_{1:p} \} = \E\{  P(T = 1 \mid U) \mid X_{1:p} \} = \E\{ e(U) \mid X_{1:p} \} = b_U(X_{1:p})$.}

\begin{align}
X_{1:p} \indep T \mid  b_U(X_{1:p}) , \label{eq:balancing projection}
\end{align}
where $b_U(X_{1:p}) := \E\{ e(U) \mid X_{1:p} \}$.
Due to \eqref{eq:balancing projection}, strict overlap in the latent variable $U$ implies strict overlap in $b_U(X_{1:p})$, which implies strict overlap in $X_{1:p}$ by Proposition~\ref{prop:sufficient}.
\citet{Athey2016} propose a specification similar to this in their simulations, in which the propensity score is dense with respect to observable covariates but can be specified simply in terms of a latent class.

\end{example}

\subsection{Outcome Models: Identification and Estimation with Weaker Overlap}
\label{sec:outcome models}

The average treatment effect can be identified and estimated under weaker overlap conditions if one is willing to make structural assumptions about the data generating process.
For example, if one assumes that the conditional expectations of outcomes $\E[Y(0) \mid X_{1:p}]$ and $\E[ Y(1) \mid X_{1:p}]$ belong to a restricted class, \citet{Hansen2008} established that $\ATE$ can be estimated under Assumption \ref{assn:unconfounded} and the following assumption.

\begin{assumption}[Prognostic Identification]
There exists some function $r(X_{1:p})$ satisfying the following two conditions 
\begin{gather}
(Y(0), Y(1)) \indep X_{1:p} \mid r(X_{1:p}), \label{eq:prognostic}\\
\eta \leq e_r(X_{1:p})  \leq 1-\eta , \label{eq:prognostic overlap}
\end{gather}
where $e_r(X_{1:p}) := P(T = 1\mid r(X_{1:p})) .$
\end{assumption}

Modifying \citet{Hansen2008}'s nomenclature slightly, we call $r(X_{1:p})$ a prognostic score.
The assumption of strict overlap in a prognostic score $r(X_{1:p})$ in \eqref{eq:prognostic overlap} is never more stringent than Assumption~\ref{assn:strict overlap} with the same $\eta$.\footnote{By the law of iterated expectations, if $\eta \leq e(X_{1:p}) \leq 1-\eta$, then
$e_{r}(X_{1:p}) = P(T = 1\mid r(X_{1:p})) = \E \{ P(T = 1\mid X_{1:p}, r(X_{1:p}))\mid r(X_{1:p}) \} 
= \E \{ P(T = 1\mid X_{1:p}) \mid r(X_{1:p}) \}\in [\eta, 1 - \eta].$ }
\citet{VanderLaan2010} and \citet{Luo2017} propose methodology designed to exploit this sort of structure.

One can also weaken overlap requirements by imposing modeling assumptions on the outcome process via the conditional average treatment effect $\tau(X_{1:p}) := \E[Y(1) - Y(0) \mid X_{1:p}]$.
If $\tau(X_{1:p})$ is assumed constant, for example, in the case of the partial linear model \citep{Belloni2014,Farrell2015}, then estimation of $\ATE$ only requires that strict overlap hold with positive probability, rather than with probability 1.
\begin{assumption}[Strict Overlap with Positive Probability]
\label{assn:constant tx}
For some $\delta > 0$,
$$
P(\eta \leq e(X_{1:p}) \leq 1-\eta) > \delta.
$$
\end{assumption}
Here, Assumption~\ref{assn:constant tx} is sufficient because the constant treatment effect assumption justifies extrapolation from subpopulations where the treatment effect can be estimated to other subpopulations for which strict overlap may fail.
The constant treatment effect assumption can also be used to justify trimming strategies, which we turn to next.%

\subsection{Trimming}
\label{sec:trimming}
When Assumption~\ref{assn:strict overlap} does not hold, one can still estimate an average treatment effect within a subpopulation in which strict overlap does hold.
This motivates the common practice of trimming, where the investigator drops observations in regions without overlap \citep{dehejia1999causal, Crump2009, Petersen2012, yang2018asymptotic}.
In general, trimming changes the estimand unless additional structure, such as a constant treatment effect, is imposed on the conditional treatment effect surface $\tau(X_{1:p})$.\footnote{An alternative strategy is to estimate a weighted average of the conditional treatment effect, e.g., $\tau^w = \E\{  w(X_{1:p}) \tau(X_{1:p}) \}$ with $ w(X_{1:p}) \propto e( X_{1:p}) \{1 - e( X_{1:p}) \}$ \citep{crump2006moving, li2018balancing}. This estimand downweights the units with propensity scores close to zero and one, and can be viewed as a smooth version of trimming. We anticipate that our argument extends to this weighting case as well.}

Our results suggest that trimming may need to be employed more often when the covariate dimension $p$ is large, especially in cases where overlap violations result from small imbalances accumulated over many dimensions.
In these cases, trimming procedures may have undesirable properties for the same reason that strict overlap does not hold.
For example, in high dimensions, one may need to trim a large proportion of units to achieve desirable overlap in the new target subpopulation.
The proportion of units that can be retained under a trimming policy designed to achieve overlap bound $\tilde \eta$ is related to the accuracy of the Bayes optimal classifier in \eqref{eq::bayes-classifier} by the following proposition. 
\begin{proposition}
For an overlap bound $\tilde{\eta} \in (0,1/2)$, we have
$$
P\left(\tilde \eta \leq e(X_{1:p}) \leq 1-\tilde \eta\right) \leq
\left[  1-P\left(\tilde\phi(X_{1:p}) = T\right)\right] \big / \tilde \eta.
$$
\begin{proof}%
Define the event $\mathcal{A} := \{\tilde \eta \leq e(X_{1:p}) \leq 1-\tilde \eta\}$.
The conclusion follows from
$$
P(\tilde \phi(X_{1:p}) \neq T) \geq P(\mathcal{A} ) P(\tilde \phi(X_{1:p}) \neq T \mid \mathcal{A} ) \geq P(\mathcal{A} ) \tilde \eta. \hfill\qedhere
$$\vspace{-2em}
\end{proof}
\end{proposition}
When large covariate sets $X_{1:p}$ enable units to be more accurately classified in treatment and control, the probability that a unit has an acceptable propensity score becomes small.
In this case, a trimming procedure must throw away a large proportion of the sample.
In the large-$p$ limit, if the Bayes optimal classifier $\tilde \phi(X_{1:p})$ is consistent in the sense of Definition~\ref{def:p-consistent}, then the expected proportion of the sample that must be discarded to achieve any $\tilde \eta$ approaches 1.

\section{Discussion}
\label{sec:future work}

In this paper, we have shown that the strict overlap assumption has strong implications in settings with high-dimensional covariates.
In particular, we show that the strict overlap assumption implies that the information distinguishing the treated and control covariate distributions must remain fixed --- even as the dimension of the covariates grows. 
This results in binding, population-level restrictions on the data-generating process.
Importantly, techniques such as regularization do not avoid these restrictions, though they are often necessary for estimation with high-dimensional covariates.

Our results suggest that overlap assumptions should be carefully considered when adjusting for rich covariates.
First, strict overlap is a testable assumption in the sense that, for any fixed bound $\eta$, one can construct finite-sample exact tests \citep{overlap_testing}.
We explore this in separate work and suggest that such empirical validation should be standard practice in these settings.
In addition, in cases where the unconfoundedness assumption is violated, overlap appears to play a key role in bias amplification phenomena that result from adjusting for covariates, such as instruments, that are highly predictive of treatment assignment but not of the outcome
\citep{Myers2011,Pearl2010,Ding2017}.
As the dimensionality increases, appropriately accounting for these complications is important both from a population and finite-sample perspective.

\clearpage

\appendix
\numberwithin{equation}{section}
\clearpage

\section{Strict Overlap Implies Bounded $f$-Divergences}
\label{sec:f bound}

Here, we adapt a theorem from information theory, due to \citet{Rukhin1997}, to derive general implications of strict overlap.
The theorem states that a likelihood ratio bound of the form \eqref{eq:bmin bmax} implies upper bounds on $f$-divergences between $P_0$ and $P_1$.
$f$-divergences are a family of discrepancy measures between probability distributions defined in terms of a convex function $f$ \citep{Csis63,Ali1966,Liese2006}.
Formally, the $f$-divergence from some probability measure $Q_0$ to another $Q_1$ is defined as
$$
D_f(Q_1(X_{1:p}) \| Q_0(X_{1:p})) := \E_{Q_0}\left[f\left(\frac{\d Q_1(X_{1:p})}{\d Q_0(X_{1:p})}\right)\right],
$$
$f$-divergences are non-negative, achieve a minimum when $Q_0 = Q_1$, and are, in general, asymmetric in their arguments.
Common examples of $f$-divergences include the Kullback--Leibler divergence, with $f(t) = t\log t$, and the $\chi^2$- or Pearson divergence, with $f(t) = (t - 1)^2$.
Here, we restate \citet{Rukhin1997}'s theorem in terms of strict overlap and the bounds defined in \eqref{eq:bmin bmax}. 

\begin{theorem}
\label{thm:f bound}
Let $D_f$ be an $f$-divergence such that $f$ has a minimum at $1$. 
Assumption \ref{assn:strict overlap} implies
\begin{align}
D_f(P_1(X_{1:p})\|P_0(X_{1:p}))
	&\leq \frac{b_{\max}-1}{b_{\max} - b_{\min}}f(b_{\min}) + 
    \frac{1-b_{\min}}{b_{\max} - b_{\min}}f(b_{\max}),\label{eq:rukhin bound}\\\notag\\
D_f(P_0(X_{1:p})\|P_1(X_{1:p}))
  &\leq \frac{b_{\min}^{-1}-1}{b_{\min}^{-1} - b_{\max}^{-1}}f(b_{\max}^{-1}) + 
  \frac{1-b_{\max}^{-1}}{b_{\min}^{-1} -
  b_{\max}^{-1}}f(b_{\min}^{-1}).\label{eq:reverse rukhin bound}
\end{align}
 \begin{proof}
 Theorem 2.1 of \citet{Rukhin1997} shows that the likelihood ratio bound in \eqref{eq:bmin bmax} implies the bounds in \eqref{eq:rukhin bound} and \eqref{eq:reverse rukhin bound} when $f$ has a minimum at $1$ and is ``bowl-shaped'', i.e., non-increasing on $(0,1)$ and non-decreasing on $(1,\infty)$.
 The ``bowl-shaped'' constraint is satisfied because $f$ is convex.
 \end{proof}
\end{theorem}

\section{Proof of Theorem~\ref{thm:general mean discrepancy}}
\subsection{Strict Overlap Implies Bounded Functional Discrepancy} 
\label{sec:functional discrepancy}
The proof of Theorem \ref{thm:general mean discrepancy} follows from several steps, each of which is of independent interest.

Here, we apply Theorem~\ref{thm:f bound} to show that strict overlap implies an upper bound on functional discrepancies of the form
\begin{align}
\left|\E_{P_0}\{ g(X_{1:p})\}  - \E_{P_1}\{ g(X_{1:p}) \} \right|\label{eq:functional discrepancy}
\end{align}
for any function $g : \mathbb R^p \mapsto \mathbb R$
that is measurable %
under $P_0$ and $P_1$.
This result plays a key role in the proof of Theorem~\ref{thm:general mean discrepancy}, but is general enough to be of independent interest.

We establish this bound by applying Theorem~\ref{thm:f bound} to the special case of the $\chi^2$-divergence
$$
\chi^2(Q_1(X_{1:p}) \| Q_0(X_{1:p})) := \E_{Q_0}\left[\left(\frac{\d Q_1(X_{1:p})}{\d Q_0(X_{1:p})} - 1\right)^2\right].
$$

Strict overlap implies the following bound on the $\chi^2$-divergence.
\begin{corollary}
\label{cor:chi square bound}
Assumption \ref{assn:strict overlap} implies  
\begin{align}
\chi^2(P_1(X_{1:p}) \| P_0(X_{1:p})) \leq  B_{\chi^2(1\|0)}, \qquad 
\chi^2(P_0(X_{1:p}) \| P_1(X_{1:p})) \leq   B_{\chi^2(0\|1)} , \label{eq:chisq10} 
\end{align}
where
$$
B_{\chi^2(1\|0)} := (1-b_{\min})(b_{\max}-1),\qquad 
B_{\chi^2(0\|1)} := (1-b_{\max}^{-1})(b_{\min}^{-1} -1). 
$$
\end{corollary}

In the case of balanced treatment assignment with $\pi = 0.5$,  
$B_{\chi^2(1\|0)}$ and $B_{\chi^2(0\|1)}$ have a simple form:
$
B_{\chi^2(1\|0)} = B_{\chi^2(0\|1)} = \{ \eta(1-\eta) \}^{-1} - 4.
$

We now apply Corollary~\ref{cor:chi square bound} to show that strict overlap implies an explicit upper bound on functional discrepancies of form \eqref{eq:functional discrepancy}. Below we let $\|g\|_{P,q} := \left[ \E_P\{ |g|^q \} \right]^{1/q}$ denote the $q$-norm of the function $g$ under measure $P$.

\begin{corollary}
\label{thm:functional discrepancy bound}
Assumption \ref{assn:strict overlap} implies
\begin{align}
\left|\E_{P_1} [g(X_{1:P})] - \E_{P_0} [g(X_{1:p})]\right| \leq
\min\Big\{
 \Var_{P_0}^{1/2}(g(X_{1:p})) \cdot
 B_{\chi^2(1\|0)}^{1/2},
 \Var_{P_1}^{1/2}(g(X_{1:p})) \cdot  B_{\chi^2(0\|1)}^{1/2}
\Big\}. \label{eq:chi2 discrepancy} 
\end{align}
\end{corollary}
\begin{proof}
By the Cauchy-Schwarz inequality, 
\begin{align}
|\E_{P_1} [g(X_{1:p})] - \E_{P_0} [g(X_{1:p})]| &=
\left|\E_{P_0} \left[(g(X_{1:p})-C) \cdot \left(\frac{\d P _1(X_{1:p})}{\d P _0(X_{1:p})} - 1\right)\right]\right|\label{eq:decomposition}\\
&\leq\|g(X_{1:p})-C\|_{P_0,2} \cdot \sqrt{\chi^2(P_1(X_{1:p})\|P_0(X_{1:p}))},
\label{eq:cauchy schwarz bound}
\end{align}
for any finite constant $C$.
A similar bound holds with respect to the $\chi^2$-divergence evaluated in the opposite direction.

Let $C = \E_{P_0} [g(X_{1:p})]$ then apply \eqref{eq:cauchy schwarz bound} and Corollary~\ref{cor:chi square bound}.
Do the same for $C = \E_{P_1} [g(X_{1:p})]$.

Corollary~\ref{thm:functional discrepancy bound} remains valid even when $\Var_{P_0}(g(X_{1:p})) = \Var_{P_1}(g(X_{1:p})) = \infty$; in this case, inequality \eqref{eq:chi2 discrepancy} holds automatically.
\end{proof}

\subsection{Proof of Theorem~\ref{thm:general mean discrepancy}}
Theorem~\ref{thm:general mean discrepancy} is a special case of Corollary~\ref{thm:functional discrepancy bound}.
In particular, let $g(X_{1:p}) := a'X_{1:p}$, where $a :=  (\mu_{1,1:p} - \mu_{0,1:p} ) / \|\mu_{1,1:p} - \mu_{0,1:p}\| $ is a vector of unit length, and apply Corollary~\ref{thm:functional discrepancy bound}.
$\Var_{P_0}(a'(X_{1:p} - \mu_{0,1:p}))$ is upper-bounded by $\|\Sigma_{0,1:p}\|_{\textup{op}}$ by definition, and likewise for $P_1$.
The result follows.

\section{Other implications of strict overlap}

The decomposition in \eqref{eq:decomposition} can be used to construct additional upper bounds on the mean discrepancy in $g$ using H\"older's inequality in combination with upper bounds on $\chi^\alpha$-divergences \citep{Vajda1973}.
These bounds give a tighter bound in terms of $\eta$, but are functions of higher-order moments of $g(X_{1:p})$. Formally, $\chi^\alpha$-divergences are a class of divergences that generalize the $\chi^2$-divergence \citep{Vajda1973}:
$$
\chi^\alpha(P_1(X_{1:p}) \| P_0(X_{1:p})) := \E_{P_0}\left[\left|\frac{\d P _1(X_{1:p})}{\d P _0(X_{1:p})}-1\right|^\alpha\right] , \quad (\alpha \geq 1).
$$
The $\chi^\alpha$ divergence in the opposite direction is obtained by switching the roles of $P_0$ and $P_1$.

Theorem 2.1 of \citet{Rukhin1997} implies that, under strict overlap with bound $\eta$,
$$
\chi^\alpha(P_0(X_{1:p}) \| P_1(X_{1:p}))  \leq
B_{\chi^\alpha(0\|1)},\qquad 
\chi^\alpha(P_1(X_{1:p}) \| P_0(X_{1:p}))  \leq
B_{\chi^{\alpha}(1\|0)}, 
$$
where 
\begin{align*}
B_{\chi^\alpha(0\|1)} &:= (b_{\max}-1)(1-b_{\min})\frac{(1-b_{\min})^{\alpha-1}+(b_{\max}-1)^{\alpha-1}}{b_{\max}-b_{\min}},\\
B_{\chi^{\alpha}(1\|0)} &:= (b_{\min}^{-1}-1)(1-b_{\max}^{-1})\frac{(1-b_{\max}^{-1})^{\alpha-1}+(b_{\min}^{-1}-1)^{\alpha-1}}{b_{\min}^{-1}-b_{\max}^{-1}} .
\end{align*}

Applying H\"older's inequality to \eqref{eq:decomposition}, we obtain
\begin{align*}
|\E_{P_1} \{  g(X_{1:p}) \}  - \E_{P_0} \{ g(X_{1:p}) \} | \leq
\min\Big\{
\|g(X_{1:p}) - C\|_{P_0,q_\alpha} \cdot
B_{\chi^\alpha(1\|0)}^{1/\alpha},
\|g(X_{1:p})-C\|_{P_1,q_\alpha}\cdot B_{\chi^\alpha(0\|1)}^{1/\alpha}\Big\}
,
\end{align*}
where $q_\alpha :=  \alpha / (\alpha-1)$ is the H\"older conjugate of $\alpha$.
Setting $C = \E_{P_0} [ g(X_{1:p}) ] $ establishes a relationship between the $q_\alpha$th central moment of $g(X_{1:p})$ under $P_0$ and the functional discrepancy between $P_0$ and $P_1$.
For small values of $\eta$, this bound scales as ${\eta}^{-1/\alpha}$, whereas \eqref{eq:chi2 discrepancy} scales as $\eta^{-1/2}$.

\section{Operator Norm}
\label{sec:operator norm}
The behavior of the bounds in Theorem~\ref{thm:general mean discrepancy} and Corollary~\ref{cor:general mean MAD} depend on the operator norm of the covariance matrix under $P_0$ and $P_1$.
Heuristically, this operator norm is large whenever there is high correlation between the covariates $X_{1:p}$ under the corresponding probability measure.
Thus, these bounds on mean imbalance become more restrictive as the dimension grows.
Because all points in this discussion apply equally to $\Sigma_{0,1:p}$ and $\Sigma_{1,1:p}$, we will refer to a generic covariance matrix $\Sigma_{1:p}$, which can be taken to be either $\Sigma_{0,1:p}$ or $\Sigma_{1,1:p}$.

In this section, we give several examples of covariance structures and the behavior of their corresponding operator norm as $p$ grows large.
In the first two examples, the operator norm is of constant order; in the third example, the growth rate of the operator norm can vary from $O(1)$ to $O(p)$.

\begin{example}[Independent Case]
\label{ex:covariance indep}
When the components of $(X_{(k)})_{k > 0}$ are independent, with component-wise variance given by $\sigma^2_k$, $\|\Sigma_{1:p}\|_{\textup{op}} = \max_{1 \leq k \leq p} \sigma^2_k$.
Thus, if the covariate-wise variances are bounded, the operator norm is $O(1)$.
\end{example}

\begin{example}[Stationary Covariance Case]
\label{ex:covariance stationary}
When $(X_{(k)})_{k > 0}$ is a stationary ergodic process with spectral density bounded by $M$, $\|\Sigma_{1:p}\|_{\textup{op}} \leq M$ \citep{Bickel2004}.
For example, when $(X_{(k)})_{k > 0}$ is an $\textsc{MA}(1)$ process with parameter $\theta$, it has a banded covariance matrix so that all elements on the diagonal $\sigma_{k,k} = \sigma^2$ and all elements on the first off-diagonal $\sigma_{k,k \pm 1} = \theta$.
In this case, the spectral density is upper bounded by $ \sigma^2(1 + \theta)^2/(2\pi)$, so the operator norm is $O(1)$.
\end{example}

\begin{example}[Restricted Rank Case]
\label{ex:covariance factor}
If $(X_{(k)})_{k > 0}$ has component-wise variances given by $\sigma^2_k$ and $\Sigma_{1:p}$ has rank $s_p$, then $\|\Sigma_{1:p}\|_{\textup{op}} \geq s_p^{-1}\sum_{k=1}^p \sigma^2_k$, because the maximum eigenvalue of $\Sigma_{1:p}$ must be larger than the average of its non-zero eigenvalues.
Thus, if $s_p = s$ is constant in $p$ and the component-wise variances are bounded away from 0 and $\infty$, the operator norm is $O(p)$.
In the special case where $s = 1$, the covariates are perfectly correlated.
On the other hand, if $s_p$ is a non-decreasing function of $p$, then the operator norm grows as $O(p/s_p)$.
\end{example}

Each example shows that if the covariates $X_{1:p}$ are not too correlated, so that $\|\Sigma_{1:p}\|_{\textup{op}} = o(p)$, strict overlap implies that the mean absolute discrepancy in \eqref{eq:general mean MAD} converges to zero, and the covariate means approach balance, on average, as $p$ grows large.

\clearpage
\singlespacing
\bibliographystyle{biometrika}
\bibliography{overlap}

\end{document}